\documentclass[12pt]{article}
\usepackage{enumerate,amsfonts,dsfont,amsmath,accents,amsthm,mathrsfs,cases,bbm,amssymb,color}
\usepackage{indentfirst} 
\usepackage[bookmarks]{hyperref}      
\usepackage{hypernat}
\usepackage{fancyhdr}
\usepackage{esint}
\usepackage{bm}

\newcounter{RomanNumber}
\newcommand{\MyRoman}[1]{\setcounter{RomanNumber}{#1}\Roman{RomanNumber}}

\allowdisplaybreaks      
\newtheorem{theorem}{Theorem}[section]
\newtheorem{lemma}[theorem]{Lemma}
\newtheorem{proposition}[theorem]{Proposition}
\newtheorem{corollary}[theorem]{Corollary}
\theoremstyle{definition}
\newtheorem{definition}[theorem]{Definition}

\newtheorem{remark}[theorem]{Remark}
\numberwithin{equation}{section}

\def\geq{\geqslant}
\def\leq{\leqslant}
\numberwithin{equation}{section}
\begin{document}
\title{The distributional hyper-Jacobian determinants  in fractional Sobolev spaces}
\author{Qiang Tu\and Chuanxi Wu\and Xueting Qiu\and
Faculty of Mathematics and Statistics, Hubei University, Wuhan 430062, China \thanks{\emph{Email addresses}:~qiangtu@whu.edu.cn(Qiang Tu),
cxwu@hubu.edu.cn (Chuanxi Wu), qiuxueting1996@163.com (Xueting Qiu).}}
\date{}   
\maketitle
\noindent{\bf Abstract:} In this paper we give a positive answer to a question raised by Baer-Jerison in connection with hyper-Jacobian determinants  and associated minors in fractional Sobolev spaces.
Inspired by recent works of Brezis-Nguyen and Baer-Jerison on the Jacobian and Hessian determinants,
we show that the distributional $m$th-Jacobian minors of degree $r$ are weak continuous in fractional Sobolev spaces $W^{m-\frac{m}{r},r}$, and
 the result is optimal, satisfying the necessary conditions, in the frame work of fractional Sobolev spaces. In particular, the conditions can be removed in case $m=1,2$, i.e.,
the $m$th-Jacobian minors of degree $r$ are well defined in $W^{s,p}$ if and only if $W^{s,p} \subseteq W^{m-\frac{m}{r},m}$ in case $m=1,2$.

\medskip

\noindent{\bf Key words:} Hyper-Jacobian,  Higher dimensional determinants, Fractional Sobolev spaces, Distributions.
\medskip

\noindent{\bf 2010 MR Subject Classification:}  46E35, 46F10,  42B35.

\section{Introduction and main results}
Fix  integer $m\geq 1$ and consider the class of non-smooth functions $u$ from $\Omega$, a smooth bounded open subset of $\mathbb{R}^N$, into $\mathbb{R}^n$( $N\geq 2$).
The aim of this article is to identify when the hyper($m$th)-Jacobian determinants and associated minors of $u$, which were introduced by Olver in \cite{O}, make sense as a distribution.

In the case $N=n$ and $m=1$,
 starting with seminal work of Morrey\cite{MC}, Reshetnyak\cite{RY} and Ball\cite{BJ} on variational problems of non-linear elasticity, it is well known that the distributional ($1$th-)Jacobian determinant $\mbox{Det}(Du)$ of a map $u\in W^{1,\frac{N^2}{N+1}}(\Omega,\mathbb{R}^N)$ (or $u\in L^{q}\cap W^{1,p}(\Omega,\mathbb{R}^N)$ with $\frac{N-1}{p}+\frac{1}{q}=1$ and $N-1< p\leq \infty$) is defined by
 $$\mbox{Det}(Du):=\sum_{j} \partial_j(u^i(\mbox{adj} Du)^i_j),$$
 where $\mbox{adj}Du$ means the adjoint matrix of $Du$.
  Furthermore, Brezis-Nguyen \cite{BN} extended the range of the map $u\mapsto \mbox{Det} (Du)$ in the framework of fractional Sobolev spaces. They showed that the distributional Jacobian determinant $\mbox{Det}(Du)$ for any $u\in W^{1-\frac{1}{N},N}(\Omega,\mathbb{R}^N)$ can be defined as
 $$\langle\mbox{Det}(Du), \psi\rangle:=\lim_{k\rightarrow \infty}\int_{\Omega}\det(Du_k)\psi dx~~~\forall \psi \in C_{c}^{1}(\Omega, \mathbb{R}),$$
where $u_k\in C^1(\overline{\Omega}, \mathbb{R}^N)$ such that $u_k\rightarrow u$ in $ W^{1-\frac{1}{N},N}$. They pointed out that  the result recovers all the definitions of distributional Jacobian determinants mentioned above, except $N=2$, and  the distributional Jacobian determinants are well-defined in $W^{s,p}$ if and only if $W^{s,p}\subseteq W^{1-\frac{1}{N},N}$ for $1<p<\infty$ and $0<s<1$.

In the case $n=1$ and $m=2$, similar to the results in \cite{BN}, the distributional Hessian(2th-Jacobian) determinants are well-defined and continuous on $W^{2-\frac{2}{N},N}(\mathbb{R}^N)$ (see \cite{IT,BJ}).   Baer-Jersion \cite{BJ} pointed out that the continuous results of  Hessian determinant  in $W^{2-\frac{2}{N},N}(\mathbb{R}^N)$ with $N\geq 3$ implies  the known continuity results  in space $W^{1,p}(\mathbb{R}^N)\cap W^{2,q}(\mathbb{R}^N)$ with $1<p,r<\infty$, $\frac{2}{p}+\frac{N-2}{q}=1$, $N\geq 3$ (see \cite{DGG,DM,FM}).
Furthermore they
showed that the distributional Hessian determinants are well-defined in $W^{s,p}$ if and only if $W^{s,p}\subseteq W^{2-\frac{2}{N},N}$ for $1<p<\infty$ and $1<s<2$.

For $m>2$, $m$th-Jacobian, as a generalization of ordinary Jacobian, was first introduced by Escherich \cite{EG} and Gegenbauer \cite{GL}.
In fact, the general formula for hyper-Jacobian can be expressed by using Cayley's theory of higher dimensional determinants.
All  these earlier investigations were limited to polynomial functions until  Olver \cite{O} turn his attention to some non-smooth functions.
 Especially he showed that the $m$th-Jacboian determinants (minors)  of degree $r$ can be defined as a distribution provided
 $$ u\in W^{m-[\frac{m}{r}],\gamma}(\Omega, \mathbb{R}^n)\cap W^{m-[\frac{m}{r}]-1,\delta} (\Omega, \mathbb{R}^n) ~\mbox{with}~\frac{r-t}{\gamma}+\frac{t}{\delta}\leq 1, t:=m ~\mbox{mod}~r$$
 or
 $$u\in W^{m-[\frac{m}{r}], \gamma}(\Omega, \mathbb{R}^n)~\mbox{with}~\gamma\geq \max\{\frac{ Nr}{N+t}\}.$$
 Bare-Jersion \cite{BJ} raised  an interesting question: whether do there exist fractional versions of this result? I.e., is the $m$th-Jacboian determinant of degree $r$  continuous from space $W^{m-\frac{m}{r},r}$ into the space of distributions? Our first results give a positive answer to the question. We refer to Sec. 2 below for the following notation.
\begin{theorem}\label{hm-thm-1}
Let $q,n,N$ be integers with $2\leq q\leq \underline{n}:=\min\{n,N\}$, for any integer $1\leq r\leq q$, multi-indices  $\beta\in I(r,n)$  and $\bm{\alpha}= (\alpha^1,\alpha^2,\cdot\cdot\cdot,\alpha^m)$ with $\alpha^j \in I(r,N)$ ($j=1,\cdots,m$), the $mth$-Jacobian $(\beta, \bm{\alpha})$-minor operator $u \longmapsto M_{\bm{\alpha}}^{\beta}(D^mu) (\mbox{see}~ (\ref{hm-pre-for-1})):C^m(\Omega,\mathbb{R}^n)\rightarrow \mathcal{D}'(\Omega)$ can be extended uniquely as a continuous mapping $u \longmapsto \mbox{Div}_{\bm{\alpha}}^{\beta}(D^mu):W^{m-\frac{m}{q},q}(\Omega,\mathbb{R}^n)\rightarrow \mathcal{D}'(\Omega)$. Moreover
for all $u,v\in W^{m-\frac{m}{q},q}(\Omega,\mathbb{R}^n)$, $\psi \in  C^{\infty}_c(\Omega,\mathbb{R})$, we have
\begin{equation}
\begin{split}
&\left|\langle\mbox{Div}_{\bm{\alpha}}^{\beta}(D^mu)-\mbox{Div}_{\bm{\alpha}}^{\beta}(D^mv),\psi\rangle\right|\\
&\leq C_{r,q,n,N,\Omega}\|u-v\|_{W^{m-\frac{m}{q},q}}\left(\|u\|_{W^{m-\frac{m}{q},q}}^{r-1} +\|v\|_{W^{m-\frac{m}{q},q}}^{r-1}\right)\|D^m\psi\|_{L^{\infty}}.
\end{split}
\end{equation}
\end{theorem}

We recall that for $0<s<\infty$ and $1\leq p<\infty$, the fractional Sobolev space $W^{s,p}(\Omega)$ is defined as follows: when $s<1$
$$W^{s,p}(\Omega):=\left\{u\in L^p(\Omega)\mid \left(\int_{\Omega}\int_{\Omega} \frac{|u(x)-u(y)|^p}{|x-y|^{N+sp}}dxdy\right)^{\frac{1}{p}}<\infty\right\},$$
and the norm
$$\|u\|_{W^{s,p}}:=\|u\|_{L^p}+\left(\int_{\Omega}\int_{\Omega} \frac{|u(x)-u(y)|^p}{|x-y|^{N+sp}}dxdy\right)^{\frac{1}{p}}.$$
When $s>1$ with non-integer,
$$W^{s,p}(\Omega):=\{u\in W^{[s],p}(\Omega)\mid D^{[s]} u\in W^{s-[s],p}(\Omega)\},$$
the norm
$$\|u\|_{W^{s,p}}:=\|u\|_{W^{[s],p}}+\left(\int_{\Omega}\int_{\Omega} \frac{|D^{[s]}u(x)-D^{[s]}u(y)|^p}{|x-y|^{N+(s-[s])p}}dxdy\right)^{\frac{1}{p}}.$$

\begin{remark}
It is worth pointing out that we may use the same method to get a similar result, see Corollary \ref{hm-cor-3-1}, for $u\in W^{m-\frac{m}{q},q}(\Omega)$ with $m\geq 2$.
Theorem \ref{hm-thm-1} and Corollary \ref{hm-cor-3-1} recover not only all the definitions of  Jacobian  and Hessian determinants mentioned above, but also the definitions of $m$-th Jacobian in \cite{O} since the following facts
\begin{enumerate}
\item [{\em(\romannumeral1)}] $W^{m-[\frac{m}{r}],\gamma}(\Omega, \mathbb{R}^n)\cap W^{m-[\frac{m}{r}]-1,\delta} (\Omega, \mathbb{R}^n)\subset W^{m-\frac{m}{r},r}(\Omega, \mathbb{R}^n)$ with continuous embedding if $\frac{r-t}{\gamma}+\frac{t}{\delta}\leq 1$ $(1<\delta<\infty, 1<r\leq N)$, where $t:=m ~\mbox{mod}~r$.
\item[{\em(\romannumeral2)}]   $W^{m-[\frac{m}{r}], \gamma}(\Omega, \mathbb{R}^n)\subset W^{m-\frac{m}{r},r}(\Omega, \mathbb{R}^n)$ $(1<r\leq N)$ with continuous embedding if $\gamma\geq \max\{\frac{ Nr}{N+t}\}$.
\end{enumerate}
\end{remark}

Similar to the optimal results for the ordinary distributional Jacobian and Hessian determinants in \cite{BN,BJ}, an natural question is that wether  the results in Theorem \ref{hm-thm-1} is  optimal in the framework of the space $W^{s,p}$? I.e., is the distributional $m$-th Jacobian minors of degree $r$ well-defined in $W^{s,p}(\Omega, \mathbb{R}^n)$ if and only if  $W^{s,p}(\Omega, \mathbb{R}^n)\subset W^{m-\frac{r}{m},r}(\Omega, \mathbb{R}^n)$?
Such a question is connected with the construction of counter-examples in some special fractional Sobolev spaces. Indeed, the above conjecture is obviously correct in case $r=1$.
Our next results give a partial positive answer  in case $r>1$.

 \begin{theorem}\label{hm-thm-2}
Let $m, r$ be  integers with $1< r\leq \underline{n}$,  $1<p<\infty$ and $0<s<\infty$ be such that $W^{s,p}(\Omega, \mathbb{R}^n) \nsubseteq W^{m-\frac{m}{r},r}(\Omega, \mathbb{R}^n)$. If the condition
 \begin{equation}\label{hm-thm-2-for-1}
 1<r<p, s=m-m/r ~\mbox{non-integer}
 \end{equation}
 fails,
 then there exist a sequence $\{u_k\}_{k=1}^{\infty} \subset C^{m}(\overline{\Omega}, \mathbb{R}^n)$,  multi-indices  $\beta\in I(r,n)$, $\bm{\alpha}=(\alpha^1,\alpha^2,\cdot\cdot\cdot,\alpha^m)$ with $\alpha^j\in I(r,N)$ and a function $\psi\in C_c^{\infty}(\Omega)$ such that
\begin{equation}\label{hm-thm-2-for-2}
\lim_{k\rightarrow \infty} \|u_k\|_{s,p} =0, ~~~~\lim_{k\rightarrow \infty} \int_{\Omega} M^{\beta}_{\bm{\alpha}}(D^mu) \psi dx=\infty,
\end{equation}
\end{theorem}
 one still unanswered question is  whether the above optimal results  hold in case (\ref{hm-thm-2-for-1}).  We give some discuss  in Sec. 4 and give  positive answers in case $m=1$ and $2$. Indeed

\begin{theorem}\label{hm-thm-3}
Let $m=1~\mbox{or}~2$ and $r,s,p$ be as in Theorem \ref{hm-thm-2}. Then there exist a sequence $\{u_k\}_{k=1}^{\infty} \subset C^{m}(\overline{\Omega}, \mathbb{R}^n)$, multi-indices  $\beta\in I(r,n)$, $\bm{\alpha}=(\alpha^1,\alpha^2,\cdot\cdot\cdot,\alpha^m)$ with $\alpha^j\in I(r,N)$ and a function $\psi\in C_c^{\infty}(\Omega)$ such that (\ref{hm-thm-2-for-2}) holds.
\end{theorem}

Furthermore, we give  reinforced versions of optimal results, see Theorem \ref{hm-thm-4-14}, for $u\in W^{2-\frac{2}{r},r}(\Omega)$ with $1<r\leq N$. we expect that there are reinforced versions of optimal results for $W^{m-\frac{m}{r},r}(\Omega)$($m>2$), for instance  there exist a sequence $\{u_k\}_{k=1}^{\infty} \subset C^{m}(\overline{\Omega})$ and a function $\psi\in C_c^{\infty}(\Omega)$ such that
\begin{equation}
\lim_{k\rightarrow \infty} \|u_k\|_{s,p} =0, ~~~~\lim_{k\rightarrow \infty} \int_{\Omega} M_{\bm{\alpha}}(D^mu) \psi dx=\infty
\end{equation}
 for any $s,p$ with $W^{s,p}(\Omega) \nsubseteq W^{m-\frac{m}{r},r}(\Omega)$.

This paper is organized as follows. Some facts and  notion about higher dimensional determinant and  hyper-Jacobian are given in Section 2.
In Section 3 we establish the  weak continuity results and definitions for distributional hyper-Jacobian minors in fractional Sobolev space.
Then we turn to the question about optimality and get some positive results in Section 4.

\section{Higher dimensional determinants}

In this section we collect some notation and preliminary results for hyper-Jacobian determinants and  minors.  Fist we recall some notation and facts about about ordinary determinants and minors, whereas further details can be found in \cite{GMS}.

Fix $0\leq k\leq n$, we shall use the standard notation for ordered multi-indices
\begin{equation}\label{subnotation01}
I(k,n):=\{\alpha=(\alpha_1,\cdot\cdot\cdot,\alpha_k) \mid \alpha_i  ~\mbox{integers}, 1\leq \alpha_1 <\cdot\cdot\cdot< \alpha_k\leq n\},
\end{equation}
where $n \geq 2$. Set $I(0,n)=\{0\}$ and $|\alpha|=k$ if $\alpha \in I(k,n)$.
For $\alpha\in I(k,n)$,
\begin{enumerate}
\item[{\em(\romannumeral1)}]   $\overline{\alpha}$ is the element in $I(n-k,n)$ which complements $\alpha$  in $\{1,2,\cdot\cdot\cdot,n\}$ in the natural increasing order.
\item [{\em(\romannumeral2)}] $\alpha-i$ means the multi-index of length $k-1$ obtained by removing $i$ from $\alpha$ for any  $i \in \alpha$.
\item [{\em(\romannumeral3)}] $\alpha+j$ means the multi-index of length $k+1$ obtained by adding j to $\alpha$ for  any $j\notin \alpha$, .
\item [{\em(\romannumeral4)}] $\sigma(\alpha,\beta)$ is the sign of the permutation which reorders $(\alpha,\beta)$ in the natural increasing order for any   multi-index  $\beta$ with $\alpha\cap \beta=\emptyset$. In particular set $\sigma(\overline{0},0):=1$.
\end{enumerate}
Let $n,N \geq 2$ and  $A=(a_{ij})_{n \times N}$ be an $n \times N$ matrix.
Given two ordered multi-indices $\alpha\in I (k,N)$ and $\beta \in I(k,n)$, then
$A_{\alpha}^{\beta}$ denotes
the $k \times k $-submatrix of $A$ obtained by selecting the rows and columns by $\beta$ and $\alpha$, respectively. Its determinant will be denoted by
\begin{equation}
M_{\alpha}^{\beta}(A):=\det A_{\alpha}^{\beta},
\end{equation}
and we set $M_{0}^{0}(A):=1$.
The adjoint of $A_{\alpha}^{\beta}$ is  defined  by the formula
$$(\mbox{adj}~ A_{\alpha}^{\beta})_j^i:= \sigma(i,\beta-i) \sigma(j,\alpha-j) \det A_{\alpha-j}^{\beta-i},~~~~ i \in \beta, j\in \alpha.$$
So Laplace  formulas can be written as
$$M_{\alpha}^{\beta}(A)= \sum_{j \in \alpha} a_{ij} (\mbox{adj}~ A_{\alpha}^{\beta})_j^i,~~~~ i\in\beta.$$

Next we pay attention to the higher dimensional matrix and determinant.

An $m$-dimensional matrix $\bm{A}$ of order $N^m$ is a hypercubical array of $N^m$  as
\begin{equation}
\bm{A}=(a_{l_1l_2\cdot\cdot\cdot l_m})_{N\times\cdot\cdot\cdot\times N},
\end{equation}
where the index  $l_i\in \{1,\cdot\cdot\cdot N\}$ for any $1\leq i\leq m$.

\begin{definition}\label{hm-def-2-1}
Let $\bm{A}$ be an $m$-dimensional matrix,  then the (full signed) determinant of $\bm{A}$ is the number
\begin{equation}
\det \bm{A}=\sum_{\tau_2,\cdot\cdot\cdot,\tau_m\in S_N} \Pi_{s=2}^m\sigma(\tau_s) a_{1\tau_2(1)\cdot\cdot\cdot\tau_m(1)} a_{2 \tau_2(2)\cdot\cdot\cdot\tau_m(2)}\cdot\cdot\cdot a_{N\tau_2(N)\cdot\cdot\cdot\tau_m(N)},
\end{equation}
where $S_N$ is the permutation group of $\{1,2,\cdot\cdot\cdot,N\}$ and $\sigma(\cdot)$ is the sign of $\cdot$.
\end{definition}

For any $1\leq i\leq m$ and $1\leq j\leq N$, the $j$-th $i$-layer of $\bm{A}$, the  $(m-1)$-dimensional matrix denoted by $\bm{A}|_{l_i=j}$, which generalizing the notion of row and column for  ordinary matrices, is defined by
$$\bm{A}|_{l_i=j}:=(a_{l_1l_2\cdot\cdot l_{i-1}jl_{i+1}\cdot\cdot\cdot l_m)})_{N\times\cdot\cdot\cdot\times N}.$$
According to Definition \ref{hm-def-2-1}, we can easily obtain that

\begin{lemma}\label{hm-lem-2-11}
Let  $\bm{A}$  be an $m$-dimensional matrix and $1\leq i\leq m$.  $\bm{A}'$ is a matrix such that a pair of $i$-layers in $\bm{A}$ is interchanged, then
$$
\det \bm{A}'=
\begin{cases}
(-1)^{m-1}\det \bm{A}~~~~~~~~i=1,\\
-\det \bm{A}~~~~~~~~ i\geq 2.
\end{cases}$$
\end{lemma}

For any $\bm{A}$ and $1\leq i<j\leq m$, the $(i,j)$-transposition of $\bm{A}$, denoting by $\bm{A}^{T(i,j)}$, is a $m$-dimensional matrix defined by
$$a'_{l_1,\cdot\cdot\cdot,l_i,\cdot\cdot\cdot,l_j,\cdot\cdot\cdot,l_m}=a_{l_1,\cdot\cdot\cdot,l_j,\cdot\cdot\cdot,l_i,\cdot\cdot\cdot,l_m}$$
for any $l_1,\cdot\cdot\cdot,l_m=1,\cdot\cdot\cdot,N$. where
$$\bm{A}^{T(i,j)}:=(a'_{l_1l_2\cdot\cdot\cdot\cdot\cdot l_m)})_{N\times\cdot\cdot\cdot\times N}.$$
Then we have
\begin{lemma}
Let  $\bm{A}$  be an $m$-dimensional matrix and $1\leq i<j\leq m$,  if $m$ is odd and $1<i<j\leq m$ or $m$ is even,  then
$$
\det \bm{A}^{T(i,j)}=\det \bm{A}.$$
\end{lemma}
\begin{proof}
According to the definition of the $m$-dimensional determinant, we only to show the claim in case $m$ is even ,$i=1$ and $j=2$.
\begin{equation*}
\begin{split}
\det \bm{A}&=\sum_{\tau_2,\cdot\cdot\cdot,\tau_m\in S_N} \Pi_{s=2}^m\sigma(\tau_s) a_{1\tau_2(1)\cdot\cdot\cdot\tau_m(1)} a_{2 \tau_2(2)\cdot\cdot\cdot\tau_m(2)}\cdot\cdot\cdot a_{N\tau_2(N)\cdot\cdot\cdot\tau_m(N)}\\
&=\sum_{\tau_2,\cdot\cdot\cdot,\tau_m\in S_N} \Pi_{s=2}^m\sigma(\tau_s) a_{\tau_2^{-1}(1)1\tau_3\circ \tau^{-1}_2(1)\cdot\cdot\cdot\tau_m\circ\tau_2^{-1}(1)} a_{\tau_2^{-1}(2) 2\tau_3\circ\tau_2^{-1}(2)\cdot\cdot\cdot\tau_m\circ\tau_2^{-1}(2)}\cdot\cdot\cdot a_{\tau_2^{-1}(N) N \tau_3\circ\tau_2^{-1}(N)\cdot\cdot\cdot\tau_m\circ \tau_2^{-1}(N)}\\
&=\sum_{\tau_2,\cdot\cdot\cdot,\tau_m\in S_N} (\sigma(\tau_2))^{m-2} \sigma(\tau^{-1}_2)\sigma(\tau_3\circ\tau^{-1}_2)\cdot\cdot\cdot \sigma(\tau_m\circ\tau^{-1}_2)\\
&\cdot a'_{1\tau_2^{-1}(1)\tau_3\circ \tau^{-1}_2(1)\cdot\cdot\cdot\tau_m\circ\tau_2^{-1}(1)} a'_{2\tau_2^{-1}(2) \tau_3\circ\tau_2^{-1}(2)\cdot\cdot\cdot\tau_m\circ\tau_2^{-1}(2)}\cdot\cdot\cdot a'_{N\tau_2^{-1}(N) \tau_3\circ\tau_2^{-1}(N)\cdot\cdot\cdot\tau_m\circ \tau_2^{-1}(N)}\\
&=\sum_{\tau'_2,\cdot\cdot\cdot,\tau'_m\in S_N} \Pi_{s=2}^m\sigma(\tau'_s) a'_{1\tau'_2\tau'_3(1)\cdot\cdot\cdot\tau'_m(1)} a'_{2\tau'_2(2) \tau'_3(2)\cdot\cdot\cdot\tau'_m(2)}\cdot\cdot\cdot a'_{N\tau'_2(N) \tau'_3(N)\cdot\cdot\cdot\tau'_m(N)}.
\end{split}
\end{equation*}
\end{proof}

More generally, suppose  $\bm{A}$ be  an $m$-dimensional matrix of order $N_1\times \cdot\cdot\cdot\times N_m$,  $1\leq r\leq \min\{N_1,\cdot\cdot\cdot,N_m\}$, and an type of multi-index $\bm{\alpha}=(\alpha^1,\alpha^2,\cdot\cdot\cdot,\alpha^m)$ where $\alpha^j:=(\alpha^j_1,\cdot\cdot,\cdot, \alpha^j_r)$, $\alpha^j_i \in \{1,2,\cdot\cdot\cdot,N_j\}$ and $\alpha^j_{i_1}\neq \alpha^j_{i_2}$ for $i_1\neq i_2$. Define the $\bm{\alpha}$-minor of $\bm{A}$, denoted by $\bm{A}_{\bm{\alpha}}$, to be the $m$-dimensional matrix of order $r^m$  as
$$\bm{A}_{\bm{\alpha}}=(b_{l_1l_2\cdot\cdot\cdot l_m})_{r\times\cdot\cdot\cdot\times r},$$
where $b_{l_1l_2\cdot\cdot\cdot l_m}:=a_{\alpha^1_{l_1}\alpha^2_{l_2}\cdot\cdot\cdot\alpha^m_{l_m}}$.  Its determinant will be denoted by
\begin{equation}
M_{\bm{\alpha}}(\bm{A}):=\det \bm{A}_{\bm{\alpha}}.
\end{equation}
If $\alpha^j$ is not increasing, let $\widetilde{\alpha^j}$ be the increasing multi-indices generated by $\alpha^j$ and $\widetilde{\bm{\alpha}}:=(\widetilde{\alpha^1},\cdot\cdot\cdot,\widetilde{\alpha^m})$, then Lemma \ref{hm-lem-2-11} implies that  $M_{\bm{\alpha}}(\bm{A})$  and $M_{\widetilde{\bm{\alpha}}}(\bm{A})$ differ only by a sign.  Without loss of generality,  we can  assume $\bm{\alpha}=(\alpha^1,\alpha^2,\cdot\cdot\cdot,\alpha^m)$ with $\alpha^j\in I(r,N_j)$. Moreover we set $M_{\bm{0}}(\bm{A}):=1$.

Next we pay attention to hyper-Jacobian determinants and minors for a  map  $u\in C^{m}(\Omega, \mathbb{R}^n)$. We will denote by $D^mu$  the hyper-Jacobian matrix of $u$, more precisely, $D^m u$  is a $(m+1)$-dimensional matrix with order $n\times N\times\cdot\cdot\cdot\times N$ given by
$$D^mu:=(a_{l_1l_2\cdot\cdot\cdot l_{m+1}})_{n\times N\times\cdot\cdot\cdot\times N}$$
where
$$a_{l_1l_2\cdot\cdot\cdot l_{m+1}}=\partial_{l_2}\partial_{l_3}\cdot\cdot\cdot\partial_{l_{m+1}} u^{l_1}.$$
Then for any $\beta\in I(r,n)$, $\bm{\alpha}=(\alpha^1,\alpha^2,\cdot\cdot\cdot,\alpha^m)$ with $\alpha^j\in I(r,N)$ and $1\leq r\leq \min\{n,N\}$, the  $m$th-Jacobian $(\beta, \bm{\alpha})$-minor of $u$, denoted by $M^{\beta}_{\bm{\alpha}}(D^mu)$, is  the determinant of the $(\beta, \bm{\alpha})$- minor of $D^mu$, i.e.,
\begin{equation}\label{hm-pre-for-1}
M^{\beta}_{\bm{\alpha}}(D^mu):=M_{(\beta,\bm{\alpha})}(D^m u).
\end{equation}
In particular if $N=n$ and $\beta=\alpha^1=\cdot\cdot\cdot=\alpha^m=\{1,2,\cdot\cdot\cdot,N\}$, $\det (D^mu)$ is called the $m$-th Jacobian determinant of $u$.
Similarly,  the hyper-Jacobian matrix $D^m u$ of $u\in C^{m}(\Omega)$ is a $m$-dimensional matrix with order $N\times\cdot\cdot\cdot\times N$ and the $m$th-Jacobian $\bm{\alpha}$-minor of $u$ is defined by $M_{\bm{\alpha}}(D^mu)$.


In order to prove the main results, some lemmas, which can be easily manipulated by the definition of hyper-Jacobian minors, are introduced as follows.

\begin{lemma}\label{hm-lem-2-2}
Let $u=(v,\cdots,v)\in C^{m}(\Omega, \mathbb{R}^n)$ with $v\in C^{m}(\Omega)$. For any $\beta\in I(r,n)$ and $\bm{\alpha}=(\alpha^1,\alpha^2,\cdot\cdot\cdot,\alpha^m)$ with $\alpha^j\in I(r,N)$, $1\leq r\leq \underline{n}$
$$M^{\beta}_{\bm{\alpha}}(D^mu)=
\begin{cases}
r! M_{\bm{\alpha}}(D^mv)~~~~m~\mbox{is even},\\
0~~~~~~~~~~~~~~~~~m~\mbox{is odd}.
\end{cases}
$$
\end{lemma}

\begin{lemma}\label{hm-lem-2-1}
Let $u\in C^{m}(\Omega, \mathbb{R}^n)$, $\beta\in I(r,n)$ and $\bm{\alpha}=(\alpha^1,\alpha^2,\cdot\cdot\cdot,\alpha^m)$ with $\alpha^j\in I(r,N)$, $1\leq r\leq \underline{n}$. Then for any $1\leq i\leq m$
\begin{equation}
M^{\beta}_{\bm{\alpha}}(D^mu)=\sum_{\tau_1,\cdot\cdot\cdot,\tau_{i-1},\tau_{i+1},\cdot\cdot\cdot,\tau_m\in S_r} \Pi_{s\in \overline{i}}\sigma(\tau_s) M^{\overline{0}}_{\alpha^i} (Dv(i)),
\end{equation}
where $M^{\overline{0}}_{\alpha^i}(\cdot)$ is the ordinary  minors and  $v(i)\in C^1(\Omega, \mathbb{R}^r)$ can be written as
$$v^j(i)=\partial_{\alpha^1_{\tau_1(j)}}\cdot\cdot\cdot \partial_{\alpha^{i-1}_{\tau_{i-1}(j)}}  \partial_{\alpha^{i+1}_{\tau_{i+1}(j)}} \cdot\cdot\cdot \partial_{\alpha^{m}_{\tau_{m}(j)}}u^{\beta_j},~~~~~~j=1,\cdots,r.$$
\end{lemma}

\section{Hyper-jacobians in fractional Sobolev spaces}
In this section we establish the weak continuity results for the Hyper-jacobian minors in the fractional Sobolev spaces $W^{m-\frac{m}{q},q}(\Omega, \mathbb{R}^n)$.

 Let $\bm{\alpha}=(\alpha^1,\alpha^2,\cdot\cdot\cdot,\alpha^m)$ with $\alpha^j \in I(r,N)$, we set
 $$\bm{\widetilde{\alpha}}=(\alpha^1+(N+1),\cdot\cdot\cdot,\alpha^m+(N+m)), R(\bm{\widetilde{\alpha}}):=\{(i_1,\cdot\cdot\cdot,i_m)\mid i_j\in \alpha^j+(N+j)\}.$$
 For any $I=(i_1,\cdot\cdot\cdot,i_m)\in R(\bm{\widetilde{\alpha}})$,
$$\widetilde{\bm{\alpha}}-I:=(\alpha^1+(N+1)-i_1,\cdot\cdot\cdot,\alpha^m+(N+m)-i_m);$$
$$\sigma(\widetilde{\bm{\alpha}}-I,I):=\Pi_{s=1}^m \sigma(\alpha^s+(N+s)-i_s,i_s);$$
$$\partial_I:=\partial_{x_{i_1}}\cdot\cdot\cdot\partial_{x_{i_m}};~~~~ \widetilde{x}:=(x_1,\cdots,x_N,x_{N+1},\cdots,x_{N+m}).$$
We begin with the following simple lemma:
\begin{lemma}\label{hm-lem-3-1}
Let $u \in C^m(\Omega, \mathbb{R}^n)$, $\psi\in C_c^m(\Omega)$, $0\leq r\leq \underline{n}:=\min\{n,N\}$, $\beta\in I(r,n)$ and  $\bm{\alpha}=(\alpha^1,\alpha^2,\cdot\cdot\cdot,\alpha^m)$ with $\alpha^j \in I(r,N)$ ($1\leq j\leq m$), then
\begin{equation}\label{hm-lem-3-1}
\int_{\Omega} M_{\bm{\alpha}}^{\beta}(D^mu)\psi dx=\sum_{I\in R(\bm{\widetilde{\alpha}})}(-1)^m\sigma(\widetilde{\bm{\alpha}}-I,I)\int_{\Omega\times [0,1)^m} M_{\widetilde{\bm{\alpha}}-I}^{\beta}(D^mU) \partial_{I} \Psi d\widetilde{x},
\end{equation}
for any extensions $U\in C^m(\Omega\times[0,1)^m,\mathbb{R}^n)\cap C^{m+1}(\Omega\times (0,1)^m,\mathbb{R}^n)$ and $\Psi\in C^m_c(\Omega\times[0,1)^m,\mathbb{R})$ of $u$ and $\psi$, respectively.
\end{lemma}
\begin{proof}
It is easy to show the results in case $r=0,1$ or $\underline{n}=1$. So we give the proof only for the case $2\leq r\leq \underline{n}$.  Denote
$$U_i:=\begin{cases}
U|_{x_{N+i+1}=\cdot\cdot\cdot=x_{N+m}=0},~~~~1\leq i\leq m-1,\\
U,~~~~i=m.
\end{cases}
\Psi_i:=\begin{cases}
\Psi|_{x_{N+i+1}=\cdot\cdot\cdot=x_{N+m}=0},~~~~1\leq i\leq m-1,\\
\Psi,~~~~i=m.
\end{cases}$$
$$\Omega_i:=\Omega\times [0,1)_{x_{N+1}}\times \cdot\cdot\cdot\times [0,1)_{x_{N+i}};~~~~\widetilde{x_i}:=(x, x_{N+1},\cdot\cdot\cdot x_{N+i}).$$
 Applying the fundamental theorem of calculus and the definition of $M_{\bm{\alpha}}^{\beta}(D^mu)$, we have
\begin{equation}\label{hm-lem-3-for-1}
\begin{split}
\int_{\Omega} M_{\bm{\alpha}}^{\beta}(D^mu)\psi dx&=-\int_{\Omega_1} \partial_{N+1} \left(M_{\bm{\alpha}}^{\beta}(D^mU_1)\Psi_1\right) d\widetilde{x_1}\\
&=-\int_{\Omega_1} \partial_{N+1}M_{\bm{\alpha}}^{\beta}(D^mU_1)\Psi_1 d\widetilde{x_1}-\int_{\Omega_1}M_{\bm{\alpha}}^{\beta}(D^mU_1)\partial_{N+1}\Psi_1 d\widetilde{x_1}.
\end{split}
\end{equation}
According to the Lemma \ref{hm-lem-2-1}, $M_{\bm{\alpha}}^{\beta}(D^mU_1)$ can be written as
\begin{equation*}
M_{\bm{\alpha}}^{\beta}(D^mU_1)= \sum_{\tau_2,\cdot\cdot\cdot,\tau_m\in S_r} \Pi_{s=2}^m\sigma(\tau_s) M^{\overline{0}}_{\alpha^1} (DV_1),
\end{equation*}
where $\overline{0}:=\{1,2,\cdot\cdot\cdot,r\}$ and
\begin{equation*}
V_1(\widetilde{x_1}):=(V_1^1(\widetilde{x_1}),\cdot\cdot\cdot,V_1^r(\widetilde{x_1})),~~~~V^j_1=\partial_{\alpha^2_{\tau_2(j)}}\cdot\cdot\cdot\partial_{\alpha^{m}_{\tau_{m}(j)}}u^{\beta_j}.
\end{equation*}
Then
\begin{equation}\label{hm-lem-3-for-12}
\int_{\Omega} M_{\bm{\alpha}}^{\beta}(D^mu)\psi dx=\sum_{\tau_2,\cdot\cdot\cdot,\tau_m\in S_r} \Pi_{s=2}^m\sigma(\tau_s) \left\{-\int_{\Omega_1} \partial_{N+1}M^{\overline{0}}_{\alpha^1} (DV_1)\Psi_1 d\widetilde{x_1}-\int_{\Omega_1}M^{\overline{0}}_{\alpha^1} (DV_1)\partial_{N+1}\Psi_1 d\widetilde{x_1}  \right\}.
\end{equation}
We denote the first part integral on the right-hand side by $I$,
Laplace formulas of the $2$-dimensional minors imply that
\begin{equation}
\begin{split}
I&=-\sum_{i\in \alpha^1} \sum_{j=1}^r \int_{\Omega_1} \sigma(i, \alpha^1-i) \sigma(j, \overline{0}-j)  \partial_{N+1}\partial_i V_1^j  M^{\overline{0}-j}_{\alpha^1-i} (DV_1)\Psi_1 d\widetilde{x_1}\\
&=\sum_{i\in \alpha^1} \sum_{j=1}^r \int_{\Omega_1} \sigma(i, \alpha^1-i) \sigma(j, \overline{0}-j)  \partial_{N+1} V_1^j \left(\partial_i M^{\overline{0}-j}_{\alpha^1-i} (DV_1)\Psi_1+ M^{\overline{0}-j}_{\alpha^1-i} (DV_1)\partial_i\Psi_1\right) d\widetilde{x_1}.\\
\end{split}
\end{equation}
Since
$$\sum_{i\in \alpha^1} \sigma(i, \alpha^1-i) \sigma(j, \overline{0}-j)\partial_i M^{\overline{0}-j}_{\alpha^1-i} (DV_1)=0$$
for any $j$, it follows that
\begin{equation}
\begin{split}
I&=\sum_{i\in \alpha^1} \sum_{j=1}^r \int_{\Omega_1} \sigma(i, \alpha^1-i) \sigma(j, \overline{0}-j)  \partial_{N+1} V_1^j M^{\overline{0}-j}_{\alpha^1-i} (DV_1)\partial_i\Psi_1 d\widetilde{x_1}\\
&=\sum_{i\in \alpha^1}  \int_{\Omega_1} \sigma(i, \alpha^1-i) \sigma(N+1, \alpha^1-i)   M^{\overline{0}}_{\alpha^1+(N+1)-i} (DV_1)\partial_i\Psi_1 d\widetilde{x_1}\\
&=-\sum_{i\in \alpha^1}  \int_{\Omega_1} \sigma(\alpha^1+(N+1)-i,i)  M^{\overline{0}}_{\alpha^1+(N+1)-i} (DV_1)\partial_i\Psi_1 d\widetilde{x_1}.
\end{split}
\end{equation}
Combing with (\ref{hm-lem-3-for-12}), we obtain that
\begin{equation*}
\int_{\Omega} M_{\bm{\alpha}}^{\beta}(D^mu)\psi dx= -\sum_{i_1\in \alpha^1+(N+1)} \sigma(\alpha^1+(N+1)-i_1,i_1) \sum_{\tau_2,\cdot\cdot\cdot,\tau_m\in S_r} \Pi_{s=2}^m\sigma(\tau_s) \int_{\Omega_1} M^{\overline{0}}_{\alpha^1+(N+1)-i_1} (DV_1)\partial_{i_1}\Psi_1 d\widetilde{x_1}.
\end{equation*}
For any $i_1\in \alpha^1+(N+1)$, we denote $\gamma:=\alpha^1+(N+1)-i_1$, then
\begin{equation}\label{hm-lem-3-for-2}
\begin{split}
&\sum_{\tau_2,\cdot\cdot\cdot,\tau_m\in S_r} \Pi_{s=2}^m\sigma(\tau_s) M^{\overline{0}}_{\alpha^1+(N+1)-i_1} (DV_1)=\sum_{\tau_1,\tau_2,\cdot\cdot\cdot,\tau_m\in S_r} \Pi_{s=1}^m\sigma(\tau_s) \partial_{\gamma_{\tau_1(1)}}V_1^1 \cdot\cdot\cdot  \partial_{\gamma_{\tau_1(r)}}V_1^r\\
&=\sum_{\tau_1,\tau_2,\cdot\cdot\cdot,\tau_m\in S_r} \Pi_{s=1}^m\sigma(\tau_s) \left(\partial_{\gamma_{\tau_1(1)}}\partial_{\alpha^2_{\tau_2(1)}} \cdot\cdot\cdot \partial_{\alpha^m_{\tau_m(1)}} U_1^{\beta_1}\right)\cdot\cdot\cdot \left( \partial_{\gamma_{\tau_1(r)}} \partial_{\alpha^2_{\tau_2(r)}} \cdot\cdot\cdot \partial_{\alpha^m_{\tau_m(r)}}  U_1^{\beta_r}\right)\\
&= M^{\beta}_{\bm{\alpha}(i_1)} (D^m U_1),
\end{split}
\end{equation}
where $\bm{\alpha}(i_1):=(\alpha^1+(N+1)-i_1,\alpha^2,\cdot\cdot\cdot,\alpha^m )$. Hence
\begin{equation}
\begin{split}
\int_{\Omega} M_{\bm{\alpha}}^{\beta}(D^mu)\psi dx&= -\sum_{i_1\in \alpha^1+(N+1)} \sigma(\alpha^1+(N+1)-i_1,i_1)  \int_{\Omega_1} M^{\beta}_{\bm{\alpha}(i_1)} (D^m U_1)\partial_{i_1}\Psi_1 d\widetilde{x_1}\\
&=\sum_{i_1\in \alpha^1+(N+1)} \sigma(\alpha^1+(N+1)-i_1,i_1)  \int_{\Omega_2} \partial_{N+2} \left( M^{\beta}_{\bm{\alpha}(i_1)} (D^m U_2)\partial_{i_1}\Psi_2\right) d\widetilde{x_2}.
\end{split}
\end{equation}
An easy induction and the  argument similar to the one used in (\ref{hm-lem-3-for-1})-(\ref{hm-lem-3-for-2}) shows that
 \begin{equation}
\begin{split}
\int_{\Omega} M_{\bm{\alpha}}^{\beta}(D^mu)\psi dx=(-1)^j\sum_{s=1}^j\sum_{i_s\in \alpha^s+(N+s)} \Pi_{s=1}^j \sigma(\alpha^s+(N+s)-i_s,i_s) \int_{\Omega_j} M_{\bm{\alpha}(i_1i_2\cdot\cdot\cdot i_j)}^{\beta}(D^mU_j)\partial_{i_1i_2\cdot\cdot\cdot i_j}\Psi_jd\widetilde{x_j}
\end{split}
\end{equation}
for any $1\leq j\leq m$,
where
$$\bm{\alpha}(i_1i_2\cdot\cdot\cdot i_j):=(\alpha^{1}+(N+1)-i_1,\cdot\cdot\cdot,\alpha^{j}+(N+j)-{i_j},\alpha^{j+1},\cdot\cdot\cdot,\alpha^{m}).$$
\end{proof}

\begin{lemma}\label{hm-lem-3-2}
Let $u,v \in C^m(\Omega, \mathbb{R}^n)$ and $\psi\in C_c^m(\Omega)$  and $2\leq q\leq  \underline{n}$. Then for any $1\leq r\leq q$, $\beta\in I(r,n)$ and $\bm{\alpha}=(\alpha^1,\alpha^2,\cdot\cdot\cdot,\alpha^m)$ with $\alpha^j \in I(r,N)$,
\begin{equation}
\left|\int_{\Omega} M_{\bm{\alpha}}^{\beta}(D^m u) \psi dx- \int_{\Omega} M_{\bm{\alpha}}^{\beta}(D^m v) \psi dx\right|\leq C\|u-v\|_{W^{m-\frac{m}{q},q}}(\|u\|^{r-1}_{W^{m-\frac{m}{q},q}}+ \|v\|^{r-1}_{W^{m-\frac{m}{q},q}}) \|D^m \psi\|_{L^{\infty}},
\end{equation}
the constant $C$ depending only on $q,r,m,n,N$ and $\Omega$.
\end{lemma}
\begin{proof}
Let $\widetilde{u}$ and $\widetilde{v}$ be extensions of $u$ and $v$ to $\mathbb{R}^N$ such that
$$\|\widetilde{u}\|_{W^{m-\frac{m}{q},q}(\mathbb{R}^N,\mathbb{R}^n)}\leq C \|u\|_{W^{m-\frac{m}{q},q}(\Omega,\mathbb{R}^n)},~~~~\|\widetilde{v}\|_{W^{m-\frac{m}{q},q}(\mathbb{R}^N,\mathbb{R}^n)}\leq C\|v\|_{W^{m-\frac{m}{q},q}(\Omega,\mathbb{R}^n)}$$
and
$$\|\widetilde{u}-\widetilde{v}\|_{W^{m-\frac{m}{q},q}(\mathbb{R}^N,\mathbb{R}^n)}\leq C \|u-v\|_{W^{m-\frac{m}{q},q}(\Omega,\mathbb{R}^n)},$$
where $C$ depending only on $q,m,n,N$ and $\Omega$.

According to  a well known trace theorem of Stein in \cite{STE1,STE2}, where $W^{m-\frac{m}{q},q}(\mathbb{R}^N)$ is identified as the space of traces of $W^{m,q}(\mathbb{R}^N\times(0,+\infty)^m)$,  there is a bounded linear extension operator
$$E:W^{m-\frac{m}{q},q}(\mathbb{R}^N,\mathbb{R}^n)\rightarrow W^{m,q}(\mathbb{R}^N\times (0,+\infty)^m,\mathbb{R}^n).$$

Let $U$ and $V$  be extensions of $\widetilde{u}$ and $\widetilde{v}$ to $\mathbb{R}^N\times (0,+\infty)^m$, respectively, i.e.,
$$U=E\widetilde{u},~~V=E\widetilde{v}.$$
We then have
$$\|D^mU\|_{L^{q}(\Omega \times (0,1)^m)}\leq C\|u\|_{W^{m-\frac{q}{m},q}(\Omega,\mathbb{R}^n)},~~~~\|D^mV\|_{L^{q}(\Omega \times (0,1)^m)}\leq C \|v\|_{W^{m-\frac{m}{q},q}(\Omega,\mathbb{R}^n)}$$
and
$$\|D^mU-D^mV\|_{L^{q}(\Omega \times (0,1)^m)}\leq C \|u-v\|_{W^{m-\frac{m}{q},q}(\Omega,\mathbb{R}^n)}.$$
Let $\Psi \in C^m_c(\Omega\times [0,1)^m)$ be an extension of $\psi$ such that
$$\|D^m\Psi\|_{L^{\infty}(\Omega\times [0,1)^m)}\leq C\|D^m\psi\|_{L^{\infty}(\Omega)}.$$
According to Lemma \ref{hm-lem-3-1}, we have
\begin{equation}\label{hm-lem-for-3-31}
\begin{split}
&\left|\int_{\Omega} M_{\bm{\alpha}}^{\beta}(D^m u) \psi dx- \int_{\Omega} M_{\bm{\alpha}}^{\beta}(D^m v) \psi dx\right|\leq \sum_{I\in R(\bm{\widetilde{\alpha}})}\int_{\Omega\times [0,1)^m} \left|M_{\widetilde{\bm{\alpha}}-I}^{\beta}(D^mU)-M_{\widetilde{\bm{\alpha}}-I}^{\beta}(D^mV)\right| |\partial_{I} \Psi| d\widetilde{x}\\
&\leq \| D^m \Psi\|_{L^{\infty}(\Omega\times [0,1)^m)} \sum_{I\in R(\bm{\widetilde{\alpha}})}\int_{\Omega\times [0,1)^m} \left|M_{\widetilde{\bm{\alpha}}-I}^{\beta}(D^mU)-M_{\widetilde{\bm{\alpha}}-I}^{\beta}(D^mV)\right| d\widetilde{x}.
\end{split}
\end{equation}
Note that for any $I\in R(\bm{\widetilde{\alpha}})$
\begin{equation*}
\begin{split}
&\left|M_{\widetilde{\bm{\alpha}}-I}^{\beta}(D^mU)-M_{\widetilde{\bm{\alpha}}-I}^{\beta}(D^mV)\right|\\
&\leq \sum_{\tau_1,\cdot\cdot\cdot, \tau_m\in S_r} |\partial_{\tau_1(1)\cdot\cdot\cdot\tau_m(1)}U^{\beta_1} \cdot\cdot\cdot \partial_{\tau_1(r)\cdot\cdot\cdot\tau_m(r)}U^{\beta_r}- \partial_{\tau_1(1)\cdot\cdot\cdot\tau_m(1)}V^{\beta_1}\cdot\cdot\cdot \partial_{\tau_1(r)\cdot\cdot\cdot\tau_m(r)}V^{\beta_r}| \\
&\leq \sum_{\tau_1,\cdot\cdot\cdot, \tau_m\in S_r}  \sum_{s=1}^{r} |D^mU|^{s-1}|D^mU-D^mV||D^mV|^{r-s}\\
&\leq C|D^mU-D^mV|(|D^mU|^{r-1}+|D^mV|^{r-1}).
\end{split}
\end{equation*}
Combining with (\ref{hm-lem-for-3-31}), we can easily obtain
\begin{equation*}
\begin{split}
&\left|\int_{\Omega} M_{\bm{\alpha}}^{\beta}(D^m u) \psi dx- \int_{\Omega} M_{\bm{\alpha}}^{\beta}(D^m v) \psi dx\right|\\
&\leq C \int_{\Omega\times [0,1)^m} |D^mU-D^mV|(|D^mU|^{r-1}+|D^mV|^{r-1}) d\widetilde{x} \|D^m\Psi\|_{L^{\infty}(\Omega\times [0,1)^m)}\\
&\leq C \|u-v\|_{W^{m-\frac{m}{q},q}}(\|u\|^{r-1}_{W^{m-\frac{m}{q},q}}+ \|v\|^{r-1}_{W^{m-\frac{m}{q},q}}) \|D^m \psi\|_{L^{\infty}}.
\end{split}
\end{equation*}
\end{proof}

 According to the above lemma,  we can give the definitions of distributional $m$th-Jacobian minors of $u$ with degree less that $q$  when $u\in W^{m-\frac{m}{q},q}(\Omega, \mathbb{R}^n)$ ($2\leq q\leq \underline{n}$).
\begin{definition}\label{hm-def-3-1}
Let $u\in W^{m-\frac{m}{q},q}(\Omega, \mathbb{R}^n)$ with $2\leq q \leq \underline{n}$. For any $0\leq r\leq q$, $\beta\in I(r,n)$ and $\bm{\alpha}=(\alpha^1,\alpha^2,\cdot\cdot\cdot,\alpha^m)$ with $\alpha^j \in I(r,N)$, the distributional  $m$th-Jacobian $(\beta, \bm{\alpha})$-minors of $u$, denoted by $\mbox{Div}_{\bm{\alpha}}^{\beta}(D^mu)$, is defined by
\begin{equation}
\langle \mbox{Div}_{\bm{\alpha}}^{\beta}(D^mu), \psi \rangle:=
\begin{cases}
\int_{\Omega} \psi(x)dx,~~~~~~~~~~~r=0; \\
\lim_{k\rightarrow \infty} \int_{\Omega} M^{\beta}_{\bm{\alpha}}(D^mu_k)\psi dx,~~~~ 1\leq r\leq q\\
\end{cases}
\end{equation}
for any $\psi\in C^m_c(\Omega)$ and any sequence $\{u_k\}_{k=1}^{\infty}\subset C^m(\overline{\Omega},\mathbb{R}^n)$ such that $u_k\rightarrow u$ in $W^{m-\frac{m}{q},q}(\Omega,\mathbb{R}^n)$.
\end{definition}
Obviously this quantity is well-defined since  Lemma \ref{hm-lem-3-2} and the fact that $C^m(\overline{\Omega}, \mathbb{R}^n)$ is dense in $W^{m-\frac{m}{q},q}(\Omega, \mathbb{R}^n)$.

\begin{proof}[\bf Proof of Theorem \ref{hm-thm-1}]
It is clear that  Theorem \ref{hm-thm-1} is a consequence of  Lemma \ref{hm-lem-3-2}  and Definition  \ref{hm-def-3-1}.
\end{proof}

According to the trace theory and the approximate theorem, we obtain a fundamental representation of the distributional m-th Jacobian minors in $W^{m-\frac{m}{q},q}$.

\begin{proposition}\label{hm-pro-3-1}
Let $u\in W^{m-\frac{m}{q},q}(\Omega, \mathbb{R}^n)$ with $2\leq q \leq \underline{n}$. For any $0\leq r\leq q$, $\beta\in I(r,n)$ and $\bm{\alpha}=(\alpha^1,\alpha^2,\cdot\cdot\cdot,\alpha^m)$ with $\alpha^j \in I(r,N)$ ,
$$\int_{\Omega} \mbox{Div}_{\bm{\alpha}}^{\beta}(D^mu)\psi dx=\sum_{I\in R(\bm{\widetilde{\alpha}})}(-1)^m\sigma(\widetilde{\bm{\alpha}}-I,I)\int_{\Omega\times [0,1)^m} M_{\widetilde{\bm{\alpha}}-I}^{\beta}(D^mU) \partial_{I} \Psi d\widetilde{x}$$
for any extensions $U\in W^{m,q}(\Omega\times[0,1)^m,\mathbb{R}^n)$ and $\Psi\in C^m_c(\Omega\times[0,1)^m)$ of $u$ and $\psi$, respectively.
\end{proposition}

Note that the $m$-dimensional matrix $D^m u$ is symmetric if $u\in C^m(\Omega)$, i.e., $(D^mu)^{T(i,j)}=D^mu$ for any $1\leq i<j\leq m$.
An argument similar to the one used in Lemma \ref{hm-lem-3-1} and \ref{hm-lem-3-2} show that
\begin{corollary}\label{hm-cor-3-1}
Let $u\in W^{m-\frac{m}{q},q}(\Omega)$ with $2\leq q \leq N$ and $m\geq 2$. For any  $0\leq r\leq q$ and $\bm{\alpha}=(\alpha^1,\alpha^2,\cdot\cdot\cdot,\alpha^m)$ with $\alpha^j \in I(r,N)$,
Then the $m$-th Jacobian $\bm{\alpha}$-minor operator $u \longmapsto M_{\bm{\alpha}}(D^mu):C^m(\Omega)\rightarrow \mathcal{D}'(\Omega)$ can be extended uniquely as a continuous mapping $u \longmapsto \mbox{Div}_{\bm{\alpha}}(D^mu):W^{m-\frac{m}{q},q}(\Omega)\rightarrow \mathcal{D}'(\Omega)$. Moreover
for all $u,v\in W^{m-\frac{m}{q},q}(\Omega)$, $\psi \in  C^{\infty}_c(\Omega,\mathbb{R})$ and $1\leq r\leq q$, we have
\begin{equation}
\begin{split}\left|\langle\mbox{Div}_{\bm{\alpha}}(D^mu)-\mbox{Div}_{\bm{\alpha}}(D^mv),\psi\rangle\right|\leq C_{r,q,N,\Omega}\|u-v\|_{W^{m-\frac{m}{q},q}}\left(\|u\|_{W^{m-\frac{m}{q},q}}^{r-1} +\|v\|_{W^{m-\frac{m}{q},q}}^{r-1}\right)\|D^m\psi\|_{L^{\infty}},
\end{split}
\end{equation}
where the constant depending only on $r, q, N$ and $ \Omega$. In particular, the distributional minor $\mbox{Div}_{\bm{\alpha}}(D^mu)$ can be expressed as
$$\int_{\Omega} \mbox{Div}_{\bm{\alpha}}(D^mu)\psi dx=\sum_{I\in R(\bm{\widetilde{\alpha}})}(-1)^m\sigma(\widetilde{\bm{\alpha}}-I,I)\int_{\Omega\times [0,1)^m} M_{\widetilde{\bm{\alpha}}-I}(D^mU) \partial_{I} \Psi d\widetilde{x}$$
for any extensions $U\in W^{m,q}(\Omega\times[0,1)^m)$ and $\Psi\in C^m_c(\Omega\times[0,1)^m)$ of $u$ and $\psi$, respectively.
\end{corollary}

\section{The optimality results in fractional Sobolev spaces}
In this section we establish the optimality results of Theorem 1 in the framework of spaces $W^{s,p}$.
Before proving the main results, we state some interesting consequences (see \cite[Theorem 1 and Proposition 5.3]{BM2}):

\begin{lemma}\label{hm-lem-4}
For $0\leq s_1<s_2<\infty$, $1\leq p_1, p_2,p\leq \infty$, $s=\theta s_1+(1-\theta) s_2$, $\frac{1}{p}=\frac{\theta}{p_1}+\frac{1-\theta}{p_2}$ and $0<\theta<1$, the inequality
$$\|f\|_{W^{s,p}(\Omega)}\leq C \|f\|_{W^{s_1,p_1}(\Omega)}^{\theta} \|f\|_{W^{s_2,p_2}(\Omega)}^{1-\theta}.$$
holds if and only if the following condition fails
$$s_2\geq 1~\mbox{is an integer}, ~p_2=1~\mbox{and}~s_2-s_1\leq 1-\frac{1}{p_1}.$$
\end{lemma}

\begin{proposition}\label{hm-pro-1}
The following equalities of spaces holds:
\begin{enumerate}
\item [{\em(\romannumeral1)}]
$W^{s,p}(\Omega)=F^s_{p,p}(\Omega)$ if $s>0$ is a non-integer and $1\leq p\leq \infty$.
\item[{\em(\romannumeral2)}]  $W^{s,p}(\Omega)=F^s_{p,2}(\Omega)$ if $s\geq 0$ is an integer and $1<p<\infty$.
\end{enumerate}
\end{proposition}
\begin{remark}
The definition of Triebel-Lizorkin spaces $F^s_{p,q}$ can be seen in \cite{BM2,TH}.
\end{remark}

\begin{remark}\label{hm-rem-41}
If $1<r\leq N$, according to the embedding properties of the Triebel-Lizorkin spaces $F^s_{p,q}$,  see e.g. \cite[page 196]{TH}, and  Proposition \ref{hm-pro-1}, we consider all possible cases:
 \begin{enumerate}
\item [{\em(\romannumeral1)}] $s-m+\frac{m}{r}>\max\{0,\frac{N}{p}-\frac{N}{r}\}$, then the embedding $W^{s,p}(\Omega)\subset W^{m-\frac{m}{r},r}(\Omega)$ holds;
\item[{\em(\romannumeral2)}]    $s-m+\frac{m}{r}<\max\{0,\frac{N}{p}-\frac{N}{r}\}$,  the embedding fails;
\item[{\em(\romannumeral3)}]    $s-m+\frac{m}{r}=\max\{0,\frac{N}{p}-\frac{N}{r}\}$, there are three sub-cases:
\begin{enumerate}
 \item[(a)] if $p\leq r$, then the embedding $W^{s,p}(\Omega)\subset W^{m-\frac{m}{r},r}(\Omega)$ holds;
 \item[(b)] if $p>r$ and $m-\frac{m}{r}$ integer, the embedding  $W^{s,p}(\Omega)\subset W^{m-\frac{m}{r},r}(\Omega)$ holds;
 \item[(c)] if $p>r$ and $m-\frac{m}{r}$  non-integer, the embedding fails.
  \end{enumerate}
\end{enumerate}
\end{remark}


In order to solve the optimality results, we just consider three cases:
\begin{equation}
\begin{split}
 &(1) 1<p\leq r, s+\frac{m}{r}<m+\frac{N}{p}-\frac{N}{r};\\
  &(2)1< r<p, 0<s<m-\frac{m}{r};\\
   &(3) 1<r<p, s=m-m/r ~\mbox{non-integer}.
   \end{split}
 \end{equation}

Without loss of generality, one may assume that $n=N$, $(-8,8)^N\subset \Omega$, and $\bm{\alpha'}=(\alpha',\cdots,\alpha')$ with $\alpha'=(1,2,\cdot\cdot\cdot,r)$.
First we establish the optimality results in case  $1< r<p, 0<s<m-\frac{m}{r}$.
\begin{proposition}\label{hm-pro-41}
Let $m, r$ be  integers with $1< r\leq \underline{n}$, $p>r$ and $0<s<m-\frac{m}{r}$. Then there exist a sequence $\{u_k\}_{k=1}^{\infty} \subset C^{m}(\overline{\Omega}, \mathbb{R}^N)$ and a function $\psi\in C_c^{\infty}(\Omega)$ such that
\begin{equation}\label{hm-thm-for-1}
\lim_{k\rightarrow \infty} \|u_k\|_{s,p} =0, ~~~~\lim_{k\rightarrow \infty} \int_{\Omega} M^{\alpha'}_{\bm{\alpha'}}(D^mu_k) \psi dx=\infty.
\end{equation}
\end{proposition}

\begin{proof}
For any integer $k$, we define $u_k: \Omega \rightarrow \mathbb{R}^N$ as
$$
u_k^i(x)= k^{-\rho} \sin (k x_i),~~1\leq i\leq r-1;~~~~
u_k^i(x)= 0,~~r< i\leq N
$$
and
$$u_k^r(x)= k^{-\rho} (x_r)^m \prod_{j=1}^{r-1} \sin (\frac{m\pi}{2}+k x_j).$$
Where $\rho$ is   a constant such that $s<\rho<m-\frac{m}{r}$. Since $\|D^{[s]+1}u_k\|_{L^{\infty}}\leq C k^{[s]+1-\rho}$ and $\|u_k\|_{L^{\infty}}\leq C k^{-\rho}$, it follows that
$$\|u_k\|_{s,p} \leq C\|u_k\|^{1-\theta}_{L^p}\|u_k\|^{\theta}_{[s]+1,p}\leq C k^{s-\rho}.$$
Where $\theta=\frac{s}{[s]+1}$.
Let $\psi\in C^{\infty}_c(\Omega)$ be such that
\begin{equation}\label{hm-th2-for-2}
\psi(x)=\prod_{i=1}^N \psi'(x_i), ~\mbox{with}~\psi'\in C^1_c((0,\pi)), \psi'\geq 0 ~\mbox{and}~\psi'=1~ \mbox{in}~(\frac{1}{4}\pi,\frac{3}{4}\pi).
\end{equation}
Then
\begin{equation*}
\int_{\Omega} M^{\alpha'}_{\bm{\alpha'}} (D^m u_k) \psi dx\geq m!\int_{(\frac{1}{4}\pi, \frac{3}{4}\pi)^N} k^{mr-\rho r-m} \prod_{j=1}^{r-1} \sin^2 (\frac{m\pi}{2}+kx_j) dx=Ck^{mr-\rho r-m}.
 \end{equation*}
 Hence the conclusion (\ref{hm-thm-for-1})  holds.
\end{proof}

Next we establishing the optimality results in case $1<r<p, s=m-m/r ~\mbox{non-integer}$ by constructing a lacunary sum of atoms,  which is inspired by the work of Brezis and Nguyen \cite{BN}.

\begin{proposition}\label{hm-pro-42}
Let $m, r$ be  integers with $1< r\leq \underline{n}$, $p>r$ and $s=m-m/r$ non-integer. Then
 there exist a sequence $\{u_k\}_{k=1}^{\infty} \subset C^{m}(\overline{\Omega}, \mathbb{R}^N)$ and a function $\psi\in C_c^{\infty}(\Omega)$ satisfying the conditions (\ref{hm-thm-for-1}).
\end{proposition}
\begin{proof}
Fix $k>>1$. Define $v_k=(v_k^1,\cdots,v_k^N):\Omega\rightarrow \mathbb{R}^N$ as follows
$$v_k^i=\begin{cases} \sum_{l=1}^k \frac{1}{n_l^{s}(l+1)^{\frac{1}{r}}} \sin( n_l x_i), ~~~~1\leq i\leq r-1;\\
 (x_r)^m\sum_{l=1}^k \frac{1}{n_l^{s}(l+1)^{\frac{1}{r}}} \prod_{j=1}^{r-1} \sin(\frac{m\pi}{2}+n_l x_j),~~~~i=r;\\
0,~~~~~~~~r+1\leq i\leq N.
\end{cases}$$
Where $n_l=k^{\frac{r^2}{m}} 8^l$ for $1\leq l\leq k$. Let $\psi\in C^{\infty}_c(\Omega)$ be defined as (\ref{hm-th2-for-2}).
We claim that
\begin{equation}\label{hm-thm-for-2}
 \|v_k\|_{s,p} \leq C,~~~~ \int_{\Omega} M^{\alpha'}_{\bm{\alpha'}}(D^mv_k) \psi dx \geq C \ln k,
 \end{equation}
 where the constant $C$ is independent of $k$.

 Assuming the claim holds, we deduce $u_k= (\ln k)^{-\frac{1}{2r}} v_k$ and $\psi$ satisfies the conditions (\ref{hm-thm-for-1}). Hence it remains to prove (\ref{hm-thm-for-2}).

On the one hand
\begin{equation}
\begin{split}
 M_{\bm{\alpha'}}^{\alpha'} (D^m v_k) &=\left\{ \prod_{i=1}^{r-1} \left( \sum_{l_i=1}^k \frac{n_{l_i}^{\frac{m}{r}}}{(l_i+1)^{\frac{1}{r}}} \sin(\frac{m\pi}{2}+n_{l_i} x_i)\right)\right\}\times \left(m! \sum_{l_r=1}^k \frac{1}{n_{l_r}^{s}(l_r+1)^{\frac{1}{r}}} \prod_{j=1}^{r-1} \sin(\frac{m\pi}{2}+n_{l_r} x_j) \right) \\
&=m! \sum_{(l_1,\cdot\cdot\cdot,l_r)\in G} \frac{1}{n_{l_r}^{s}(l_r+1)^{\frac{1}{r}}}  \prod_{i=1}^{r-1} \left(\frac{n_{l_i}^{\frac{m}{r}}}{(l_i+1)^{\frac{1}{r}}} \sin(\frac{m\pi}{2}+n_{l_i} x_i) \sin(\frac{m\pi}{2}+n_{l_r} x_i) \right)\\
&+m! \sum_{l=1}^k \frac{1}{l+1}  \prod_{i=1}^{r-1} \sin^2(\frac{m\pi}{2}+n_l x_i),
   \end{split}
 \end{equation}
where
$$G:=\{(l_1,\cdot\cdot\cdot,l_r)\mid (l_1,\cdot\cdot\cdot,l_r)\neq (l,\cdot\cdot\cdot,l) ~\mbox{for}~l,l_1,\cdot\cdot\cdot,l_r=1,\cdot\cdot\cdot,k\}.$$

Hence
\begin{equation}\label{hm-thm-for-6}
\int_{\Omega} M^{\beta}_{\bm{\alpha'}}(D^mv_k) \psi dx\geq  C \sum_{l=1}^k \frac{1}{l+1}  \int_{(\frac{1}{4}\pi, \frac{3}{4}\pi)^N}  \prod_{i=1}^{r-1} \sin^2(\frac{m\pi}{2}+n_l x_i) dx-CI,
\end{equation}
where
$$I:= \left| \int_{\Omega} \psi(x) \sum_{(l_1,\cdot\cdot\cdot,l_r)\in G} \frac{1}{n_{l_r}^{s}(l_r+1)^{\frac{1}{r}}}  \prod_{i=1}^{r-1} \left(\frac{n_{l_i}^{\frac{m}{r}}}{(l_i+1)^{\frac{1}{r}}} \sin(\frac{m\pi}{2}+n_{l_i} x_i) \sin(\frac{m\pi}{2}+n_{l_r} x_i) \right) dx\right|.$$
Since $n_l=k^{\frac{r^2}{m}} 8^l$, it follows that
\begin{equation}\label{hm-thm-for-3}
\frac{n_{l_i}}{n_{l_j}}\leq |n_{l_i}-n_{l_j}|~\mbox{for any}~l_i,l_j=1,\cdot\cdot\cdot,k~\mbox{with}~l_i\neq l_j,
\end{equation}
\begin{equation}\label{hm-thm-for-4}
\min_{i\neq j} |n_{l_i}-n_{l_j}|\geq k^{\frac{r^2}{m(r-1)}}
\end{equation}
and
\begin{equation}\label{hm-thm-for-5}
\{n_l\mid l=1,\cdot\cdot\cdot,k\}\cap \{z\in \mathbb{R}\mid 2^{n-1}\leq |z|< 2^n\}~\mbox{has at most one element for any }~n\in \mathbb{N}.
\end{equation}
For any $(l_1,\cdot\cdot\cdot,l_r)\in G$, there exists $1\leq i_0\leq  r-1$ such that $l_{i_0}\neq l_r$, it follows from  (\ref{hm-th2-for-2}), (\ref{hm-thm-for-3}) and (\ref{hm-thm-for-4})  that
\begin{equation}
\begin{split}
&\left|\frac{1}{n_{l_r}^{s}(l_r+1)^{\frac{1}{r}}} \int_{\Omega} \psi(x) \prod_{i=1}^{r-1} \left(\frac{n_{l_i}^{\frac{m}{r}}}{(l_i+1)^{\frac{1}{r}}} \sin(\frac{m\pi}{2}+n_{l_i} x_i) \sin(\frac{m\pi}{2}+n_{l_r} x_i) \right) dx\right|\\
&\leq  \frac{C}{n_{l_r}^{s}(l_r+1)^{\frac{1}{r}}}   \prod_{i=1}^{r-1} \frac{n_{l_i}^{\frac{m}{r}}}{(l_i+1)^{\frac{1}{r}}} \left|\int_0^{\pi} \psi'(x_i)  \sin(\frac{m\pi}{2}+n_{l_i} x_i) \sin(\frac{m\pi}{2}+n_{l_r} x_i) dx_i \right|\\
&\leq C  \prod_{i=1}^{r-1} \left(\frac{n_{l_i}}{n_{l_r}}\right)^{\frac{m}{r}} \min \{\frac{1}{|n_{l_i}-n_{l_r}|^m},1\} \|D^m\psi\|_{L^{\infty}}\\
&\leq  \frac{C}{|n_{l_{i_0}}-n_{l_r}|^{m-\frac{m}{r}}}\\
&\leq C k^{-r}.
   \end{split}
 \end{equation}
Combine with (\ref{hm-thm-for-6}), we find
\begin{equation}
\int_{\Omega} M^{\alpha'}_{\bm{\alpha'}}(D^mv_k) \psi dx\geq  C \sum_{l=1}^k \frac{1}{l+1}  -C,
\end{equation}
which implies the second inequality of (\ref{hm-thm-for-2}). On the other hand, in order to prove the first inequality of (\ref{hm-thm-for-2}), it is enough to show that
\begin{equation}\label{hm-thm-for-10}
\|v'_k\|_{s,p} \leq C,
\end{equation}
where $v'_k:=(v^1_k,v^2_k,\cdot\cdot\cdot,v^{r-1}_k,\frac{v^r_k}{(x_r)^m})$.
In fact, the Littlewood-Paley characterization of the Besov space $B^s_{p,p}([0,2\pi]^N)$ (e.g. \cite{TH}) implies that
\begin{equation}\label{hm-thm-for-11}
\|v'_k\|_{s,p}\leq C \left(\|v'_k\|^p_{L^p([0,2\pi]^N)}+\sum_{j=1}^{\infty} 2^{sjp}\|T_j(v'_k)\|^p_{L^p([0,2\pi]^N)}\right)^{\frac{1}{p}}.
\end{equation}
Here the bounded operators $T_j:L^p\rightarrow L^p$ are defined by
$$T_j\left(\sum a_n e^{in\cdot x}\right)=\sum_{2^j\leq |n|< 2^{j+1}} \left( \rho(\frac{|n|}{2^{j+1}})- \rho(\frac{|n|}{2^{j}})\right)a_ne^{in\cdot x},$$
 where $\rho\in C_c^{\infty}(\mathbb{R})$ is a suitably chosen bump function. Then we have
 \begin{equation}\label{hm-thm-for-12}
 \|T_j(v'_k)\|^p_{L^p([0,2\pi]^N)}\leq C_p \sum_{l=1}^k \frac{1}{n_l^{sp}(l+1)^{\frac{p}{r}}} \|T_j(g_{l,k})\|^p_{L^p([0,2\pi]^N)},
 \end{equation}
where $g_{l,k}=(\sin(n_l x_1),\cdot\cdot\cdot,\sin(n_l x_{r-1}),\prod_{j=1}^{r-1} \sin(\frac{m\pi}{2}+n_l x_j))$. Indeed, since $\sin(n_l x_i)=\frac{1}{2i}(e^{in_lx_i}-e^{-in_lx_i})$, $g_{l,k}$ can be written as
$$g_{l,k}(x)= \sum_{\varepsilon\in \{-1,0,1\}^{r-1}} a_{\varepsilon} e^{n_l i \varepsilon\cdot \widehat{x}},$$
where $\widehat{x}=(x_1,\cdot\cdot\cdot,x_{r-1})$, $|a_{\varepsilon}|\leq 1$ for any $\varepsilon$. Set
$$S(j,l)=\{\varepsilon\in \{-1,0,1\}^{r-1}\mid 2^{j-1}\leq n_l |\varepsilon|<2^{j+2}\}$$
and
$$\chi(j,l)=\begin{cases}
1~~~~~S(j,l)\neq \emptyset\\
0~~~~~S(j,l)= \emptyset
\end{cases}.$$
Hence
\begin{equation}\label{hm-thm-for-13}
\|T_j(g_{l,k})\|^p_{L^p([0,2\pi]^N)}\leq C_{r,N}  \chi(j,l).
\end{equation}
For any $j$, if $S(j,l)\neq \emptyset$, then $\frac{2^{j-1}}{\sqrt{r-1}}\leq n_l< 2^{j+2}$, which implies that $\sum_{l=1}^{k}\chi(j,l)<[\frac{\log_2(r-1)}{6}]+1$. Thus, applying (\ref{hm-thm-for-11}), (\ref{hm-thm-for-12}) and (\ref{hm-thm-for-13}), we have
\begin{equation}
\begin{split}
\|v'_k\|_{s,p}^p&\leq  C_{p,s,N,r}\left( \|v'_k\|^p_{L^p([0,2\pi]^N)}+\sum_{j=1}^{\infty} \sum_{l=1}^k\frac{2^{sjp}}{n_l^{sp}(l+1)^{\frac{p}{r}}}\chi(j,l)\right)\\
&\leq C_{p,s,N,r}\left( \|v'_k\|^p_{L^p([0,2\pi]^N)}+ \sum_{l=1}^k\frac{1}{(l+1)^{\frac{p}{r}}} \left(\sum_{j=1}^{\infty}\chi(j,l)\right)\right).
\end{split}
\end{equation}
which implies (\ref{hm-thm-for-10}) since  $\sum_{j=1}^{\infty}\chi(j,l)\leq  [\frac{\log_2(r-1)}{2}]+4$ for any $l$.
\end{proof}

\begin{proof}[\bf Proof of Theorem \ref{hm-thm-2}]
Clearly Theorem \ref{hm-thm-2} is a consequence of  Proposition \ref{hm-pro-41}  and  \ref{hm-pro-42} as explained in Remark \ref{hm-rem-41}.
\end{proof}

Next we pay attention to the optimality results in case $1<p\leq r, s+\frac{m}{r}<m+\frac{N}{p}-\frac{N}{r}$.
\begin{proposition}\label{hm-pro-43}
Let $m, r$ be  integers with $1<p\leq r\leq \underline{n}$ and $s+\frac{m}{r}<m+\frac{N}{p}-\frac{N}{r}$. If there exist a function $g\in C_c^{\infty}(B(0,1), \mathbb{R}^n)$, $\beta\in I(r,n)$ and $\bm{\alpha}=(\alpha^1,\alpha^2,\cdot\cdot\cdot,\alpha^m)$ with $\alpha^j \in I(r,N)$ such that
\begin{equation}\label{hm-thm-for-15}
\int_{B(0,1)} M_{\bm{\alpha}}^{\beta} (D^m g(x)) |x|^m dx  \neq 0.
\end{equation}
 Then
 there exist a sequence $\{u_k\}_{k=1}^{\infty} \subset C^{m}(\overline{\Omega}, \mathbb{R}^N)$ and a function $\psi\in C_c^{\infty}(\Omega)$ satisfying the conclusions (\ref{hm-thm-2}).
\end{proposition}
\begin{proof}
For any $0<\varepsilon<<1$ we set
\begin{equation}
u_{\varepsilon}=\varepsilon^\rho g(\frac{x}{\varepsilon}),
\end{equation}
where $\rho$ is a constant such that $s-\frac{N}{p}<\rho<m-\frac{N}{r}-\frac{m}{r}$.

On the one hand, Lemma \ref{hm-lem-4} implies that
\begin{equation}
\|u_{\varepsilon}\|_{s,p}\leq C\|u_{\varepsilon}\|^{\theta}_{L^p}\|u_{\varepsilon}\|^{1-\theta}_{[s]+1,p}\leq C\varepsilon^{\rho+\frac{N}{p}-s}\|g\|^{\theta}_{L^p}\|D^{[s]+1}g\|_{L^p}^{1-\theta},
\end{equation}
where $\theta=\frac{[s]+1-s}{[s]+1}$.
On the other hand, let $\psi\in C^{\infty}_c(\Omega)$ be such that $\psi(x)= |x|^m+ O(|x|^{m+1})$ as $x\rightarrow 0$. Then
\begin{equation}
\begin{split}
&\int_{\Omega} M_{\bm{\alpha}}^{\beta} (D^m u_{\varepsilon}) \psi dx=\varepsilon^{\rho r-rm+N}  \int_{B(0,1)} M_{\bm{\alpha}}^{\beta}(D^m g(x)) \psi(\varepsilon x) dx\\
&=\varepsilon^{\rho r -rm+N+m} \int_{B(0,1)} M_{\alpha}^{\beta}(D^mg(x)) |x|^m dx +O( \varepsilon^{\rho r-rm+N+m+1}).
   \end{split}
 \end{equation}
Take $\varepsilon = \frac{1}{k}$ and hence the conclusion is proved.
\end{proof}

In order to establishing the optimality results in case $1<p\leq r, s+\frac{m}{r}<m+\frac{N}{p}-\frac{N}{r}$, a  natural problem is raised whether there exists $g\in C_c^{\infty}(B(0,1), \mathbb{R}^N)$ such that   the conclusion (\ref{hm-thm-for-15}) holds.
We  have  positive answers to the problem in case $m=1$ or $2$, see Theorem \ref{hm-thm-3}, according the following Lemma:

\begin{lemma}\label{hm-lem-3}
Let $g\in C_c^{\infty}(B(0,1))$ be given as
\begin{equation}\label{hm-lem-for-1}
g(x)=\int_0^{|x|} h(\rho) d\rho
\end{equation}
for any $x\in \mathbb{R}^N$,  where $h\in C_c^{\infty}((0,1))$ and satisfies
$$\int_0^1 h(\rho)d\rho=0,~~~~\int_0^1h^{r}(\rho)\rho^{-r+N+s-1} d\rho\neq 0.$$
Here $r\geq 2,s\geq 1$ are integers.
Then for any $\alpha \in I(r,N)$, we have
\begin{equation}\label{hm-lem-for-2}
\int_{B(0,1)} M_{\alpha}^{\alpha}( D^2 g(x)) |x|^s dx\neq 0.
\end{equation}
\end{lemma}
\begin{proof}
 It is easy to see that
$$D^2 g=\frac{1}{|x|^3}(A+B),$$
where $A=(a_{ij})_{N\times N}$ and $B=(b_{ij})_{N\times N}$ are $N\times N$ matrices such that
$$a_{ij}=h(|x|)|x|^2\delta_{i}^{j},~~b_{ij}=\left(h'(|x|)|x|-h(|x|)\right)x_ix_j,~~~~i,j=1,\ldots,N.$$
Using Binet formula and the fact $\mbox{rank}(B)=1$, one has
\begin{align*}
M_{\alpha}^{\alpha}(A+B)&=M_{\alpha}^{\alpha}(A)+\sum_{i\in \alpha}\sum_{j\in\alpha}\sigma(i,\alpha-i)\sigma(j,\alpha-j)b_{ij}M_{\alpha-i}^{\alpha-j}(A)\\
&=h^r(|x|)|x|^{2r} - h^r(|x|)|x|^{2r-2}\sum_{i\in\alpha} x_i^2+h^{r-1}(|x|)h'(|x|)|x|^{2r-1} \sum_{i\in\alpha} x_i^2,
\end{align*}
Hence
\begin{align*}
\int_{B(0,1)}M_{\alpha}^{\alpha}( D^2 g)|x|^s dx=\int_{B(0,1)} |x|^{-3r+s}M_{\alpha}^{\alpha}(A+B) dx=\MyRoman{1}-\MyRoman{2}+\MyRoman{3},
\end{align*}
where
$$\MyRoman{1}:=\int_{B(0,1)} h^r(|x|) |x|^{-r+s} dx,$$
$$\MyRoman{2}:=\int_{B(0,1)} h^r(|x|) |x|^{-r-2+s} \sum_{i\in \alpha} x_i^2dx,$$
and
$$\MyRoman{3}:=\int_{B(0,1)} h^{r-1}(|x|)h'(|x|)|x|^{-r-1+s} \sum_{i\in \alpha} x_i^2dx.$$
Then integration in polar coordinates gives
$$\MyRoman{3}=\frac{r-N-s}{N} 2\pi\prod_{i=1}^{N-2}I(i) \int_0^1 h^r(\rho) \rho^{-r+N+s-1} d\rho,$$
where $I(i)=\int_0^{\pi} \sin^i \theta d\theta$. Similarly,
$$\MyRoman{2}=\frac{r}{N} 2\pi\prod_{i=1}^{N-2}I(i) \int_0^1 h^r(\rho) r^{-r+N+s-1} d\rho,$$
and
$$\MyRoman{1}= 2\pi\prod_{i=1}^{N-2}I(i) \int_0^1 h^r(\rho) \rho^{-r+N+s-1} d\rho,$$
which implies (\ref{hm-lem-for-2}), and then the proof is complete.
\end{proof}

\begin{proof}[\bf Proof of Theorem \ref{hm-thm-3}]

Note that if $m=2$ and $g= (g',\cdots,g')$ with $g'\in C^{2}(\Omega)$, then Lemma \ref{hm-lem-2-2} implies
$$M_{\bm{\alpha}}^{\alpha}(D^2g)=r! M_{\alpha^1}^{\alpha^2}(D^2g')$$
for any $\bm{\alpha}=(\alpha^1,\alpha^2)$, $\alpha\in I(r,N)$.
Hence Theorem \ref{hm-thm-3} is the consequence of Proposition \ref{hm-pro-41}, \ref{hm-pro-42},  \ref{hm-pro-43} and Lemma \ref{hm-lem-3}.
\end{proof}

In particular, we can give a reinforced versions of optimal results in case $m=2$.
\begin{theorem}\label{hm-thm-4-14}
Let $1< r\leq N$, $1<p<\infty$ and $0<s<\infty$ be such that $W^{s,p}(\Omega) \nsubseteq W^{2-\frac{2}{r},r}(\Omega)$. Then there exist a sequence $\{u_k\}_{k=1}^{\infty} \subset C^{m}(\overline{\Omega})$ and a function $\psi\in C_c^{\infty}(\Omega)$ such that
\begin{equation}
\lim_{k\rightarrow \infty} \|u_k\|_{s,p} =0, ~~~~\lim_{k\rightarrow \infty} \int_{\Omega} M^{\alpha'}_{\alpha'}(D^2u_k) \psi dx=\infty.
\end{equation}
\end{theorem}

\begin{proof}
We divide our proof in three case:

\textbf{Case 1:} $1<p\leq r$ and $s+\frac{2}{r}<2+\frac{N}{p}-\frac{N}{r}$

Apply Lemma  \ref{hm-lem-3} and the argument similar to one used in Proposition \ref{hm-pro-43}.

\textbf{Case 2:} $r<p$ and $0<s<2-\frac{2}{r}$

For $k>>1$, we set
$$u_k:=k^{-\rho}  x_{r}\Pi_{i=1}^{r-1} \sin^2(kx_i),$$
where $\rho$ is a constant with $s<\rho<2-\frac{2}{r}$.
According to the facts that $\|u_k\|_{L^{\infty}}\leq C k^{-\rho}$ and $\|D^2u_k\|_{L^{\infty}}\leq C k^{2-\rho}$, it follows that
$$\|u_k\|_{s,p}\leq C \|u_k\|_{L^p}^{1-\frac{s}{2}} \|u_k\|_{2,p}^{\frac{s}{2}}\leq C k^{s-\rho}.$$
On the other hand, Let $\psi\in C^{\infty}_c(\Omega)$ be defined as (\ref{hm-th2-for-2}), the (4.1) in \cite[Proposition 4.1]{BJ}  implies that
\begin{equation}
\begin{split}
&\left |\int_{\Omega} M_{\alpha'}^{\alpha'} (D^2 u_k) \psi dx\right|\geq   \left| \int_{(\frac{1}{4}\pi, \frac{3}{4}\pi)^N} M_{\alpha'}^{\alpha'}(D^2u_{k})dx\right|\\
&\geq    k^{2r-2-r\rho} 2^r \int_{(\frac{1}{4}\pi, \frac{3}{4}\pi)^N}  x_r^{r-2} \left(\prod_{i=1}^{r-1} \sin(kx_i)\right)^{2r-2}   \left(\sum_{j=1}^{r-1} \cos^2(kx_j)\right)dx\\
&=Ck^{2r-2-r\rho}.
   \end{split}
 \end{equation}

\textbf{Case 3:}  $2<r<p$ and $s=2-\frac{2}{r}$

For any $k\in \mathbb{N}$ with $k\geq 2$,
define $u_k$ with
$$u_k(x)=\frac{1}{(\ln k)^{\frac{1}{2r}}} x_r \sum_{l=1}^k \frac{1}{n_l^{2-\frac{2}{r}} l^{\frac{1}{r}}} \prod_{i=1}^{r-1} \sin^2 (n_l x_i)~~~~x\in \mathbb{R}^N,$$
where $n_l=k^{r^{3l}}$.
Let $\psi\in C^{\infty}_c(\Omega)$ be defined as (\ref{hm-th2-for-2}).
The argument similar to the one used in \cite[Proposition 5.1]{BJ} shows that
$$\|u_k\|_{W^{s,p}(\Omega)}\leq C \|u_k\|_{W^{s,p}((0,2\pi)^N)}\leq C \frac{1}{(\ln k)^{\frac{1}{2r}}}$$
and
$$
\left| \int_{\Omega} M_{\alpha'}^{\alpha'}(D^2u_{k}) \psi dx\right|=C \left| \int_{(0,2\pi)^{r}} M_{\alpha'}^{\alpha'}(D^2u_{k}) \prod_{i=1}^r \psi'(x_i) dx_1\cdot\cdot\cdot dx_r\right|
\geq C (\ln k)^{\frac{1}{2}}.
$$
\end{proof}


\section*{Acknowledgments}
\addcontentsline{toc}{chapter}{Acknowledgements}
This work is supported by NSF grant of China ( No. 11131005, No. 11301400) and Hubei Key Laboratory of Applied Mathematics (Hubei University).

\bibliographystyle{plain}

\end{document}